\documentclass{amsart}

\usepackage[utf8]{inputenc}
\usepackage{epsfig}
\usepackage{amssymb}
\usepackage{amscd}
\usepackage[matrix,arrow]{xy}
\usepackage{graphicx}
\usepackage{amsmath, amsthm}
\usepackage{amsfonts, bbm, mathrsfs}
\usepackage{color, enumerate}
\usepackage{soul}
\usepackage{hyperref}
\hypersetup{colorlinks=true,citecolor=NavyBlue,linkcolor=Brown,urlcolor=Orange}
\usepackage[usenames,dvipsnames,svgnames,table]{xcolor, pstricks}
\usepackage{tikz, multicol}
\usepackage{tikz-cd}
\usepackage{float}
\usepackage{mathtools}
\usepackage{comment}
\usepackage{enumitem}

\newtheorem*{thm*}{Theorem}
\newtheorem*{prop*}{Proposition}
\newtheorem*{cor*}{Corollary}
\newtheorem{thm}{Theorem}[section]

\newtheorem{lemma}[thm]{Lemma}

\numberwithin{equation}{section}

\theoremstyle{definition}

\newtheorem{defn}[thm]{Definition}

\newtheorem{rmk}[thm]{Remark}

\def\bR{\mathbb{R}}
\def\bC{\mathbb{C}}
\def\bL{\mathbb{L}}
\def\bZ{\mathbb{Z}}
\def\bQ{\mathbb{Q}}
\def\bP{\mathbb{P}}

\def\bL{\mathbb{L}}
\def\bS{\mathbb{S}}

\def\be{\mathbf{e}}
\def\bf{\mathbf{f}}

\def\cD{\mathcal{D}}
\def\cF{\mathcal{F}}
\def\cK{\mathcal{K}}
\def\cO{\mathcal{O}}

\def\cP{\mathcal{P}}

\def\cQ{\mathcal{Q}}

\def\Stab{\mathrm{Stab}}
\def\Aut{\mathrm{Aut}}
\def\ch{\mathrm{ch}}
\def\td{\mathrm{td}}
\def\Hom{\mathrm{Hom}}

\def\GL{\mathrm{GL}}
\def\SL{\mathrm{SL}}
\def\SO{\mathrm{SO}}

\def\NS{\mathrm{NS}}
\def\Coh{\mathrm{Coh}}
\def\Fuk{\mathrm{Fuk}}

\def\Pic{\mathrm{Pic}}

\def\Muk{\mathrm{Muk}}
\def\Lag{\mathrm{Lag}}
\def\sLag{\mathrm{sLag}}
\def\SLag{\mathrm{SLag}}
\def\SL{\mathrm{SL}}
\def\re{\mathrm{Re}}
\def\im{\mathrm{Im}}
\def\Vol{\mathrm{Vol}}
\def\disc{\mathrm{Disc}}
\def\vol{\mathrm{vol}}

\def\rk{\mathrm{rk}}
\def\Amp{\mathrm{Amp}}

\def\ra{\rightarrow}

\def\l{\langle}
\def\r{\rangle}

\title[Counting special Lagrangian classes and semistable Mukai vectors]{Counting special Lagrangian classes and semistable Mukai vectors for K3 surfaces}
\author{Jayadev~S.~Athreya, Yu-Wei~Fan, Heather~Lee}
\date{}

\begin{document}

\begin{abstract} Motivated by the study of the growth rate of the number of geodesics in flat surfaces with bounded lengths, we study generalizations of such problems for K3 surfaces. In one generalization, we give a result regarding the upper bound on the asymptotics of the number of classes of irreducible special Lagrangians in K3 surfaces with bounded period integrals.  In another generalization, we give the exact leading term in the asymptotics of the number of Mukai vectors of semistable coherent sheaves on algebraic K3 surfaces with bounded central charges, with respect to generic Bridgeland stability conditions.  
\end{abstract}
\maketitle

\setcounter{tocdepth}{2}
\tableofcontents

\section{Introduction}
\subsection{Motivation and problems} For a Riemann surface, a holomorphic 1-form $\Omega$ (also called an Abelian differential) with finitely many zeros endows the surface a flat metric given by $g=\frac{1}{2}\Omega \overline \Omega$, with conical singularities at the zeros of $\Omega$, where the cone angles are integer multiples of $2\pi$.  The Abelian differential also gives an area form $\omega=\frac{i}{2}\Omega\wedge \overline \Omega.$ Such flat surfaces have nontrivial but restricted holonomy that the parallel transport of a vector along a small closed path around a conical singularity turns the vector by exactly the cone angle. Studying flat surfaces reveals many interesting mathematical phenomena as reviewed in \cite{Zorich}. We are interested in studying generalizations of problems on flat surfaces to other settings.

Having a metric, it is natural to study closed geodesics and saddle connections, which are geodesic segments connecting pairs of conical points. A geodesics can be described as a curve on which $\im(e^{i\phi}\Omega)$ restricts to zero, for some angle $\phi$.  Indeed, in a local coordinate chart away from the zeros of $\Omega$, we can express $\Omega$ as $\Omega=dz$, then straight lines of angle $\phi$ are the loci where  $\im(e^{i\phi}z)=0$. For a flat surface $(S, \Omega)$ with area normalized to one, one of the interesting problems is to understand the growth rate, as $R\to \infty$, of the counting function $N_{S,\Omega}(R)$ of the number of saddle connections, or the number of maximal cylinders filled with closed geodesics, of length at most $R$. This has been intensively studied. Masur \cite{Masur} showed the existence of quadratic upper and lower bounds for $N_{S,\Omega}(R)$ {for all flat surfaces}, and Eskin-Masur~\cite{EskinMasur}, building on work of Veech~\cite{Veech}, showed the exact quadratic asymptotics for almost all flat surfaces with respect to the natural Masur-Smillie-Veech measure on the space of quadratic differentials. More recently, Eskin-Mirazkahani-Mohammadi~\cite{EMM} showed {a Ces\`aro-type quadratic asymptotic} for every flat surface.

One example of generalization of geodesics on flat surfaces is to consider special Lagrangian submanifolds  in Calabi-Yau manifolds. A Calabi-Yau manifold of complex dimension $n$ is a complex K\"ahler manifold that admits a nowhere vanishing holomorphic top form (i.e. $n$-form).  By \cite{Yau78}, a compact Calabi-Yau manifold admits a Ricci-flat K\"ahler metric.  While this is not an exact higher dimensional generalization of flat surfaces, it provides a class of examples in a similar spirit since {the Ricci-flat metric} also has nontrivial but restricted holonomy.  In this paper, we specifically look at K3 surfaces.  A K3 surface is a compact complex surface that admits a nowhere vanishing holomorphic $2$-form $\Omega$ and is simply connected. By \cite{Siu83}, all K3 surfaces are K\"ahler, so they are Calabi-Yau. 

 Given a Calabi-Yau manifold $X$ of real dimension $2n$ with a Ricci-flat K\"ahler form $\omega$, this symplectic form $\omega$ is used to define Lagrangian submanifolds, which are real $n$ dimensional submanifolds on which $\omega$ restricts to 0. (On a Riemann surface, any curve is a Lagrangian.) The holomorphic top form $\Omega$ is used in addition to define special Lagrangian submanifolds similar to how it is used to define geodesics on Riemann surfaces.  A Lagrangian submanifold $L$ is special of phase $\phi$ if $\im(e^{i\phi}\Omega)|_L=0$.  Analogous to the counting function of geodesics on a flat surface, here we can consider a lattice theoretic counting function of irreducible special Lagrangian (sLag) classes with bounded period integral
\begin{equation}\label{eq: SL intro}
SL_{\omega,\Omega}(R)=\#\left\{\gamma\in H^n(X,\bZ)\colon \exists \text{ irreducible sLag } L \text{ s.t. } [L]^{Pd}=\gamma, |\gamma.\Omega|\leq R \right\},
\end{equation}
where $[L]^{Pd}$ denotes the Poincare dual of the homology class of $L$, and $\gamma.\Omega$ denotes the intersection pairing which is the period integral $\int_L\Omega$. In this paper, we focus on K3 surfaces where $n=2$. 

{Let $X$ be a  K3 surface with a Ricci-flat metric $g$.
There is a $2$-sphere family $(X,J_t)$, $t\in \bS^2$, known as the twistor family, of K3 surfaces that are all compatible with the metric $g$. 
Their associated K\"ahler forms  $\omega_t$, and the real and imaginary parts of their (normalized) nowhere vanishing holomorphic $2$-forms $\Omega_t$, all lie in a positive definite $3$-plane $P \subset H^2(X,\bR)$.}
The counting function of interest for special Lagrangian classes in this twistor family $\bS^2(P)$ formulation is  
\begin{multline}\label{eq: SL P intro}
SL_P(R)=\#\left\{\gamma\in H^2(X,\bZ)\colon \exists \omega_t \in \bS^2(P), L\  \text{irreducible sLag}  \right. \\ 
\text{w.r.t. }(\omega_t, \Omega_t), \left. [L]^{Pd}=\gamma, \vert\gamma.\Omega_t\vert \leq R\right\}.
\end{multline}
As pointed out in \cite{FilipK3twistor}, the Riemann surface analogue of the twistor sphere  is a circle in a cylinder, for which rotating along a circle does not change the metric nor the complex structure. The difference for a K3 surface is that the complex structure changes as we vary along the twistor sphere.  The work of \cite{FilipK3twistor, physicsCount} study the count of special Lagrangian tori in K3 surfaces in this twistor formulation, and \cite{physicsCount} also explains the relevance to physics.  Very briefly, special Lagrangians in K3 correspond to BPS states via compactifying  M-theory on K3, and their count is relevant to understanding the Bekenstein-Hawking entropy of simple black holes via counting microstates \cite{StromingerVafa}. 

Another generalization of geodesics that we can consider are semistable coherent sheaves, which correspond to Lagrangian submanifolds via  homological mirror symmetry (HMS).  The HMS conjecture was first formulated by \cite{Kontsevich} to fully capture the phenomenon of mirror symmetry, which is a nontrivial duality between complex geometry and symplectic geometry, as an equivalence of triangulated categories.  For a mirror pair of Calabi-Yau manifolds $(X,\omega_X, J_X)$ and $(Y,\omega_Y, J_Y)$, it predicts the following equivalences
\begin{equation}
D^\pi\Fuk(X,\omega_X)\cong D^b\Coh(Y,J_Y) \quad \text{and}\quad D^b\Coh(X,\omega_X)\cong D^\pi\Fuk(Y,J_Y), 
\end{equation}
where $D^\pi\Fuk$ is the split-closed derived Fukaya category and $D^b\Coh$ is the bounded derived category of coherent sheaves. For K3 surfaces, HMS is proved in the case of Greene-Plesser mirrors by \cite{SheridanSmithK3}.  

The Fukaya category depends on the symplectic structure but not on the complex structure, and its objects are Lagrangian submanifolds (other objects that might not be geometric are added in the derived category).  As we mentioned in the above, if in addition we consider a complex structure and a holomorphic top form, we can define special Lagrangian submanifolds that are the stable objects.  On the other hand, the category of coherent sheaves depends on the complex structure but not on the symplectic structure.  If in addition we consider the symplectic structure, we can define stable coherent sheaves.  {For example,  for a vector bundle $E$ on a complex curve, its slope is defined to be $\mu(E):=\deg E/\mathrm{rk}  E$, and $E$ is stable (semistable) if every subbundle $F$ satisfies $\mu(F)<\mu(E)$ ($\mu(F)\leq \mu(E)$).}  These considerations are behind Bridgeland's definition of stability conditions
building on Douglas' work \cite{Douglas01, DouglasICM} on $\Pi$-stability for D-branes.

The space of Bridgeland stability condition $\Stab(\cD)$ is defined \cite{BridgelandAnnals} on a triangulated category $\cD$, e.g. the derived category of coherent sheaves $\cD=\cD^b\Coh(X)$ for a Calabi-Yau manifold $X$. When $X$ is an algebraic K3 surface and $\cD=\cD^b\Coh(X)$, Bridgeland \cite{BridgelandK3} gives a detailed description of one connected component of $\Stab(\cD)$. Denote by $K(\cD)$ its Grothendieck group consisting of objects in $\cD$ up to some equivalence relations given by the triangulated structure.  For an element $E\in K(\cD)$, its Mukai vector $v(E)$ is a cohomology class element in $H^*(X,\bZ)$ (see  Equation \eqref{eqn: mukai vector} for a more precise definition).  The set of Mukai vectors are in bijection with the numerical Grothendieck group $N(\cD)$, which is the quotient $K(\cD)/\ker \chi(-,-)$, where $\chi$ is the Euler pairing.  The space of Bridgeland stability conditions $\Stab(\cD)$ consists of all pairs {$\sigma=(Z_\sigma, \cP_\sigma)$, where $Z_\sigma: N(\cD)\to \bC$ is a group homomorphism  called the central charge and $\cP_\sigma(\phi)$} is a full additive subcategory of $\cD$ consisting of $\sigma$-semistable objects of phase $\phi$ for each $\phi\in \bR$ (together with a few axioms that these need to satisfy). In this setting, the counting function of interest is the number of $\sigma$-semistable Mukai vectors with bounded central charge
\begin{equation}\label{eq: N_sigma intro}
N_\sigma(R)=\#\{\gamma\in N(\cD) \colon|Z_\sigma(\gamma)|\leq R\text{ and }
\exists\ \sigma\text{-semistable object }E\text{ with } v(E)=\gamma\}.
\end{equation}

We summarize in Table \ref{table:analogyflatsurface} the analogies between geodesics and saddle connections on flat Riemann surfaces, special Lagrangians in Calabi-Yau manifolds, and semistable objects in triangulated categories.  We summarize in Table \ref{table:analogymirrorK3} the correspondence between a mirror pair of K3 surfaces.  
\begin{table}[ht]
\centering
\setlength\extrarowheight{2pt}
    \begin{tabular}{c|c|c}
    Flat surfaces & Calabi--Yau manifolds & Triangulated categories \\
    \hline\hline
    Abelian differentials & Holomorphic top forms & Stability conditions \\ 
    \hline
    Geodesics & Special Lagrangians & Semistable objects \\
    \hline
    Lengths & Period integrals & Central charges 
    \end{tabular}
\caption{Analogy with flat surfaces}
    \label{table:analogyflatsurface}
\end{table}
\begin{table}[ht]
\centering
\setlength\extrarowheight{2pt}
    \begin{tabular}{c|c}
    K3 surface $X$ & K3 surface $Y$ \\
    \hline\hline
    \begin{tabular}{@{}c@{}}Symplectic form $\omega_X$, \\ which defines Lagrangian  \end{tabular}& \begin{tabular}{@{}c@{}}Holomorphic $2$-form $\Omega_Y$, \\ which gives $\cD^b\Coh(Y,\Omega_Y)$\end{tabular} \\
    \hline
    \begin{tabular}{@{}c@{}}Holomorphic $2$-form $\Omega_X$, \\ which defines special Lagrangian \end{tabular}&
     \begin{tabular}{@{}c@{}}Symplectic form $\omega_Y$, \\ which gives $\sigma_\omega\in\Stab(\cD^b(Y))$\end{tabular} \\
    \hline
    $\int_{\text{--}}\Omega_X$ & $Z_{\sigma_\omega}(\text{--})=(\exp(\omega),\text{--})_\Muk$
    \end{tabular}
\caption{Correspondence between mirror pairs of K3 surfaces $(X,Y)$}
\label{table:analogymirrorK3}
\end{table}


\subsection{Results} We begin by stating our result on the growth rate of of the number of semistable classes $N_\sigma(R)$ as introduced in Equation \eqref{eq: N_sigma intro}.   In general, \cite{BridgelandAnnals} showed that $\Stab(\cD)$ is a complex manifold, with a left action by the group of triangulated autoequivalences  $\Aut(\cD)$. 
There is a subset $U(\cD)\subset \Stab(\cD)$ known as the geometric stability conditions (see Definition \ref{def: geometric stab}), and denote by $\Stab^\dagger(\cD)$ the connected component of $\Stab(\cD)$ that contains $U(\cD)$.  For $\cD:=D^b\Coh(X)$ being the  the derived category of coherent sheaves on an algebraic K3 surface,  \cite{BridgelandK3} showed that there is some autoequivalence of $\cD$ that maps $\Stab(\cD)^\dagger$ into the closure of a subset $U(\cD)$.
Therefore, for almost every $\sigma\in\Stab^\dagger(\cD)$, we can find a representative $\sigma'\in U(\cD)$ such that $N_\sigma(R)=N_{\sigma'}(R)$. Furthermore, there is a subset $V(\cD)\subset U(\cD)$ that is the set of geometric stability conditions of phase 1 (more precisely defined in Equation \eqref{eq: def V}).  It is shown in \cite{BridgelandK3} that for any $\sigma' \in U(\cD)$, $Z_{\sigma'}$ can be obtained from $Z_{\sigma}$, for some $\sigma\in V(\cD)$, via a $\GL^+(2,\bR)$ action.  More precisely, there is a $g\in\GL^+(2,\bR)$ such that  $Z_{\sigma'}=Z_\sigma g$, where $g$ acts by multiplication to the image of $Z_\sigma$ in $\bC\cong \bR^2$. Any element of $\GL^+(2,\bR)=\bR^+\times\SL(2,\bR)$ is a composition of three different types of elements: 
\begin{itemize}
    \item  $g\in \bR^+$,
    \item  $g\in \SO(2,\bR)$, i.e. in the rotational part of $\SL(2,\bR)$, and 
    \item  $g=\begin{pmatrix}1&\kappa\\ 0&\lambda\end{pmatrix}$ is a shear by $\kappa+i\lambda$.
\end{itemize}
For a geometric stability of phase 1, its central charge is of the form   $Z_\sigma=\exp(B+i\omega)$, where $B,\omega$ are real divisor classes and $\omega$ is ample.

\begin{thm}\label{thm: count stable objects}
Let $X$ be a complex algebraic K3 surface. Then for almost every $\sigma\in\Stab^\dagger(\cD^b(X))$,
\[
N_\sigma(R)=C(\sigma)\cdot R^{\rho+2}+o(R^{\rho+2}),
\]
where $\rho$ is the Picard rank of $X$, and $C(\sigma)$ is a constant depending on $\sigma$. For $\sigma \in V(\cD)$,
\[C(\sigma)=\frac{2\pi^{(\rho+2)/2}}{(\rho+2)\Gamma(\frac{\rho}{2}+1)(\omega^2)^{(\rho+2)/2}\sqrt{|\disc\NS(X)|}}.
\] 
For $\sigma'\in U(\cD)$, $Z_{\sigma'}=Z_{\sigma}g$ for some $\sigma \in V(\cD)$ and $g\in \GL^+(2,\bR)$.  There are three cases
\begin{itemize}
\item  $g\in \bR^+$, then  $C(\sigma')=\frac{C(\sigma)}{g^{\rho+2}}$.
\item  $g$ is in the rotation part of $\SL(2,\bR)$, then $C(\sigma')=C(\sigma)$.
\item $g$ is a shear by $\kappa+i\lambda$, $\lambda>0$, then
\[
C(\sigma')=\frac{\pi^{\rho/2}\displaystyle\int_0^{2\pi}   \left(\cos^2\theta+ \frac{1}{\lambda^2} (\sin\theta-\kappa\cos\theta)^2\right)^{\rho/2}d\theta}{(\rho+2)\Gamma\left(\frac{\rho}{2}+1\right)\lambda(\omega^2)^{(\rho+2)/2}\sqrt{|\mathrm{Disc}\NS(X)|}}.
\]
\end{itemize}
\end{thm}
\noindent One can see that the coefficient $C(\sigma)$ of the leading term does not depend on $B$.  As stated in the theorem, this formula applies to ``almost every'', or generic, stability condition, which make up the complement of a measure zero set (more precisely, see Definition \ref{def: generic} and the paragraph below that).

Now we turn to the counting function $SL_{\omega,\Omega}(R)$ of special Lagrangian classes as introduced in Equation \eqref{eq: SL intro} (with $n=2$ for K3 surfaces).  Denote by $\Lag(X,\omega):=H^2(X,\bZ)\cap [\omega]^\perp\subset H^2(X,\bZ)$ be the lattice of consisting of cohomology classes of Lagrangian submanifolds (see the paragraph containing Equation \eqref{eq: Lag def} for more explanation).  We have the following result.  

\begin{thm}\label{thm: SL omega Omega}
Assuming $\Lag(X,\omega)$ has signature $(2,19)$, then
\[
SL_{\omega,\Omega}(R)\leq
C(\omega,\Omega)R^{21}+o(R^{21})\]
where 
\[
C(\omega,\Omega)=\frac{2\pi^{21/2}}{21\Gamma(\frac{21}{2})K_{\Omega}^{{21}/2}\sqrt{|\disc \Lag(X,\omega)|}}.
\]
\end{thm}
\noindent Note that $\Lag(X,\omega)$ has signature $(2,19)$ if $\omega$ is a rational K\"ahler class, so the statement holds in particular for polarized K3 surfaces.

In the twistor formulation as introduced in Equation \eqref{eq: SL P intro}, we have the following result.
\begin{thm}\label{thm: SL P}
Let $P$ be a positive definite $3$-plane in $H^2(X,\bR)$ {parametrizing a twistor family $\{(X,\omega_t,\Omega_t)\}_{t\in\bS^2(P)}$}. Then
\[
SL_P(R)\leq C\cdot R^{22}+o(R^{22})
\]
where $C$ is a constant independent of the choice of the twistor family $\bS^2(P)$.
\end{thm}

\noindent\textbf{Acknowledgement:} We would like to thank the Mathematical Sciences Research Institute for its hospitality in Fall 2019, where this work originated at the program on \emph{Holomorphic Differentials in Mathematics and Physics}.  We also want to thank Monty McGovern for helpful conversations, and the anonymous referee for their careful reading and helpful comments.

\section{Background}
\subsection{Lattice terminology}
A lattice $\bL$ is a finitely generated free $\bZ$-module, together with a nondegenerate symmetric bilinear form $(-, -): \bL \times \bL \to \bZ$.   With respect to a basis on the $\bZ$-module, the bilinear form can be represented by a matrix.  The signature $(n_+, n_-)$ of $\bL$ is the number $n_+$ of positive and the number $n_-$ of negative eigenvalues of that matrix.  The signature is independent of the choice of {basis}.

\subsection{K3 surfaces}
A K3 surface is a compact complex surface $X$ that admits a nowhere vanishing holomorphic $2$-form and is simply connected.  Denote the complex structure of $X$ by $J$, and let $\Omega_J$ be a nowhere vanishing holomorphic 2-form, whose choice is unique up to scaling.   In some occasions in this paper, we would like to only consider algebraic K3 surfaces, which coincides with the subset of all complex analytic K3 surfaces that are projective. 
\subsubsection{K3 lattice}
All K3 surfaces are diffeormorphic, so they have the same cohomology groups, which are  $H^0(X,\bZ)\cong H^4(X,\bZ)\cong \bZ$, $H^1(X,\bZ)\cong H^3(X,\bZ)=0$, and $H^2(X,\bZ)\cong \bZ^{22}$. Because $H^2(X,\bZ)$ is free, the  intersection pairing,  
\begin{equation}
( - . -): H^2(X,\bZ)\times H^2(X, \bZ) \to \bZ,
\end{equation}
given by the cup product is a nondegenerate symmetric bilinear form.  We will always use $a.b$ or $(a.b)$ to denote intersection pairing.

A K3 lattice is $H^2(X,\bZ)$ together with the intersection pairing above.  For any K3 surface, there is a choice of basis for $H^2(X,\bZ)$ with respect to which the intersection pairing corresponds to a block matrix of the form 
\begin{equation}
\Lambda :=U^{\oplus 3}\oplus (-E_8)^{\oplus 2},
\end{equation}
where $U$ represents $\bZ^2\cong \bZ\langle e \rangle \oplus \bZ\langle f\rangle$ with the bilinear form given by $e^2=f^2=0$ and $e.f=1$.   So $U$ is called the hyperbolic plane, and it has signature $(1,1)$.  
The $E_8$ lattice is positive definite, so altogether $\Lambda$ has signature $(3, 19)$.  An {isometry}
\begin{equation}
\varphi: H^2(X,\bZ)\to \Lambda
\end{equation}
induced by choosing a basis is called a marking of $X$.

\subsubsection{Weight-two Hodge structure and the period domain} The intersection pairing on $H^2(X,\bZ)$ extends complex linearly to the intersection pairing on $H^2(X,\bC)=H^2(X,\bZ)\otimes_\bZ \bC$ that is the same as
\begin{equation}
(-.-): H^2(X,\bC)\otimes H^2(X,\bC)\to \bC, \ \ \ \ \ a.b =\int_{X}a\wedge b.
\end{equation}
The isometry $\varphi:H^2(X,\bZ)\to \Lambda$ extends to 
\begin{equation}
\varphi_\bC := \varphi \otimes \bC: H^2(X,\bC)\to \Lambda_\bC:=\Lambda\otimes_\bZ\bC.
\end{equation}

Because K3 surfaces are K\"ahler \cite{Siu83}, we have the following weight-two Hodge decomposition
\begin{equation}\label{eqn: hodge structure}
H^2(X, \bC)= H^{2,0}(X) \oplus H^{1,1}(X)\oplus  H^{0,2}(X).
\end{equation}
This decomposition is determined by the complex line  {$H^{2,0}(X)\cong \bC[\Omega_J]$}.  This is because  $H^{0,2}(X)$ is complex conjugate to $H^{2,0}(X)$, and $H^{1,1}(X)$ is orthogonal to $H^{2,0}(X)\oplus H^{0,2}(X)$.   We have $h^{2,0}(X)=h^{0,2}(X)=1$ and  $h^{1,1}(X)=20$.

 Two K3 surfaces, $X_1$ and $X_2$, are said to be Hodge isometric if there is an isomorphism $\phi:H^2(X_1,\bZ)\to H^2(X_2,\bZ)$ such that $\phi$ preserves the intersection form and $(\phi\otimes \bC)(H^{2,0}(X_1))=H^{2,0}(X_2)$, hence preserving the weight-two Hodge structure. Torelli theorem for K3 surfaces  says that $X_1$ and $X_2$, are isomorphic as complex manifolds if and only if they are Hodge isometric.  {This is proved by} \cite{TorelliAG} {in the algebraic case and by} \cite{TorelliNonAG} {in the non-algebraic case, and see} \cite[Chapter 2]{HuybrechtsBook} {for an exposition. }

Note that $\Omega_J$ satisfies $\Omega_J.\Omega_J=0$ and $\Omega_J.\overline\Omega_J>0$. We call the following set  the period domain,
\begin{equation}\label{eq: period domain}
\begin{array}{ll}
\Omega(\Lambda_{\bC})&:= \{ [\Omega]\in \bP(\Lambda_\bC): \Omega.\Omega=0, \ \Omega. \overline \Omega>0\} \\
 &\ = \{ [\Omega]\in \bP(\Lambda_\bC): K_\Omega = (\re \Omega)^2=(\im \Omega)^2>0, (\re \Omega).(\im\Omega) =0\}.
  \end{array}
\end{equation}
{The map $X\mapsto [\varphi_\bC(H^{2,0}(X))]\in \Omega (\Lambda_\bC)$ } is called the period mapping, and it takes a $(2,0)$-form $\Omega$ to its marking $\varphi_\bC(\Omega)\in \Lambda_\bC$, which corresponds to the line $[\varphi_\bC(\Omega)]\in \bP(\Lambda_\bC)$. This map is surjective but not injective {on the moduli space of marked K3 surfaces} \cite[Chapter 6]{HuybrechtsBook}.

\subsubsection{N\'eron-Severi lattice}

For compact K\"ahler manifolds, the image of the first Chern class map $c_1: \Pic(X)\to H^2(X,\bC)$ is known as the N\'eron-Severi group, $\NS(X)$, i.e. {numerical classes} of line bundles classified by the first Chern class.  By Lefschetz theorem on $(1,1)$-classes, this image is equal to
\begin{equation}
H^{1,1}(X,\bZ):= H^{1,1}(X)\cap \mathrm{Image} \left(H^2(X,\bZ)\to H^2(X,\bC)\right).
\end{equation}
The kernel of this map is $H^1(X,\cO_X)$, which is $0$ for K3 surfaces, so this map is injective.  Hence for K3 surfaces, the $c_1$ map gives the identification
\begin{equation}
\Pic(X) = \NS(X) = H^{1,1}(X,\bZ).
\end{equation}
Because $H^{1,1}(X)$ has dimension 20, we get that the Picard rank, $\rho(X) := \rk(\Pic(X))$, is in the range $0\leq \rho(X)\leq 20$. For projective K3 surfaces, $1\leq \rho(X)\leq 20$.

Hodge index theorem says that for a compact K\"ahler surface, the intersection pairing 
\begin{equation}
H^2(X,\bR)\times H^2(X,\bR)\to \bR, \ \ \ a.b \mapsto \int_X a\wedge b
\end{equation}
restricted to $H^{1,1}(X,\bR)$ has signature $(1, h^{1,1}(X)-1)$.   One can also define an intersection pairing 
\begin{equation} \label{eq: NSpairing}
\Pic(X) \times \Pic(X)\to \bZ, \ \ \ L.L' =\int_X c_1(L)\wedge c_1(L').
\end{equation}
A line bundle $L\in \Pic(X)$  is called numerically trivial if $L. L'=0$ for all line bundles $L'$.  For complex projective surfaces, a line bundle is numerically trivial if and only if $c_1(L)=0$.  So for projective K3 surfaces, since the kernel of $c_1$ is trivial, {only the trivial bundle is numerically trivial}.  Hence for projective K3 surfaces,  the intersection pairing \eqref{eq: NSpairing} on the N\'eron-Severi group is nondegenerate and has signature $(1, \rho(X)-1)$.

\subsubsection{Twistor family}\label{sec: twistor} Let $(X,J)$ be a K3 surface, by \cite{Yau78}, any K\"ahler class in $H^{1,1}(X,\bR)\subset H^2(X,\bR)\cong  \Lambda_\bR$ has a unique representative $\omega$ whose associated metric $g(\cdot, \cdot)=\omega_J(\cdot, J\cdot)$ is Ricci-flat, {i.e.~with vanishing Ricci curvature.}

Denote by $\Omega$ a nowhere vanishing holomorphic 2-form. Note that  $\re\Omega$ and $\im \Omega$ are orthogonal to $\omega$ because $\Omega\in H^{2,0}(X)$, $\overline \Omega\in H^{0,2}(X)$, and $H^{2,0}(X)\oplus H^{0,2}(X)$ is orthogonal to $H^{1,1}(X)$.  Also, as in Equation \eqref{eq: period domain}, $(\re\Omega)^2=(\im\Omega)^2>0$ and $\re\Omega.\im\Omega=0$. So
\begin{equation}\label{eq: P}
P:=\mathrm{Span}\{\re\Omega, \im\Omega, \omega\}\subset H^2(X,\bR)\cong \Lambda_\bR
\end{equation}
is a positive-definite 3-plane in $\Lambda_\bR\cong \bR^{3,19}$.  If we assume the normalization $d\vol:=\frac{1}{2}\Omega\wedge \overline{\Omega}$ and $\int_Xd\vol=1$, then this is  equivalent to  $K_\Omega=(\re \Omega)^2=(\im \Omega)^2=1$.  If furthermore the K\"ahler class $[\omega]$ we have satisfies  $\int_X\omega^2=1$, then $\{\re\Omega, \im\Omega, \omega\}\subset \bS^2(P)$ form an orthonormal basis for $P$.


The Ricci-flatness means that the holonomy group of $g$ is in $\mathrm{SU}(2)$.  Because $\mathrm{SU}(2)$ is isomorphic to the group $\mathrm{Sp}(1)$ of unitary quaternions, $g$ is a hyperk\"ahler metric.  Consequently, $X$ admits a two-sphere $\bS^2$ family, known as the twistor family, of complex structures that are compatible with $g$ and we now describe them.  Denote by $\omega_J$ and $\Omega_J$ the forms on $(X,J)$, and assume they are normalized as above.  Let $\omega_K=\re(\Omega_J)$ and $\omega_I=\im(\Omega_J)$.  Let $I$ be the complex structure on $X$ that is compatible with $(g,\omega_I)$, i.e. $\omega_I(\cdot,\cdot)=g(I\cdot,\cdot)$.  Similarly, let $K$ be the complex structure compatible with $(g, \omega_K)$.
With respect to $I$, $\Omega_I=\omega_J+i\omega_K$ is a holomorphic 2-form.  Similarly, with respect to $K$, $\Omega_K= \omega_I+i\omega_J$ is a holomorphic 2-form.   The complex structures $I, J, K$ satisfies the quaternionic commutation relation 
$I^2=J^2=K^2=I\cdot J\cdot K =-1$. Then, any complex structure $J_t$, $t=(x, y, z)$,  in the following two sphere
\begin{equation}
\{J_t=xI +y J + zK\  | \ x^2+y^2+z^2=1\} =\bS^2
\end{equation}
is compatible with the metric $g$, i.e. $g(\cdot, \cdot)=g(J_t \cdot, J_t\cdot)$.  For each $t\in \bS^2$, denote by $\omega_t$ the associated K\"ahler form given by  $\omega_t(\cdot,\cdot)=g(J_t\cdot, \cdot)$, and denote by $\Omega_t$ a normalized nowhere vanishing holomorphic 2-form with respect to $J_t$ {(there is a $S^1$-worth of compatible choices for such $\Omega_t$)}.  Then $\{\omega_t, \re\Omega_t, \im\Omega_t\}\subset \bS^2(P)$, for $P$ in Equation \eqref{eq: P} and $P$ is known as the twistor plane.
 
Conversely, given any positive definite 3-plane $P\subset H^2(X,\bR)$, there is an associated twistor family on $X$.

\subsection{Coherent sheaves} In this section, we assume $X$ is projective.  Denote by $\cD:= \cD^b \Coh(X)$ the bounded derived category of coherent sheaves on $X$. 
\subsubsection{Mukai lattice}
  Denote by $K(\cD)$ the Grothendieck group, which admits an Euler pairing 
\begin{equation}
\chi(E, F) :=\sum\limits_j (-1)^j \dim \Hom_{\cD}^j(E, F).
\end{equation}
Serre duality implies that the Euler pairing is symmetric:
$\chi(E,F)=\chi(F,E)$.

The Mukai vector is a map 
 \begin{equation}
 v: K(\cD) \to H^*(X,\bZ)=H^0(X,\bZ)\oplus H^2(X,\bZ)\oplus H^4(X,\bZ)
 \end{equation}
  given by
\begin{equation}\label{eqn: mukai vector}
v(E): =\ch(E)\sqrt{\td(X)}  =  (\rk(E), \ c_1(E), \ \chi(E)-\rk(E)).
\end{equation}
The Hirzebruch-Riemann-Roch formula implies
\begin{equation} \label{eqn: euler mukai}
\chi(E, F) = -\l v(E), v(F)\r,
\end{equation}
where 
\begin{equation}
\begin{array}{c}
\l-,-\r: H^*(X,\bZ)\times H^*(X,\bZ)\to \bZ,  \\
\l (r_1, D_1, s_1), (r_2, D_2, s_2)\r=  D_1. D_2- r_1.s_2-r_2.s_1
\end{array}
\end{equation}
is the Mukai pairing.  Notation-wise, we will always use $(-.-)$ to denote the intersection pairing and $\l-,-\r$ to denote the Mukai pairing.   Note that the Mukai pairing on $H^*(X,\bZ)$  restricted to $H^2(X,\bZ)$ is the same as the intersection pairing.  

We call $H^*(X,\bZ)$, equipped with the Mukai pairing, the Mukai lattice.  According to (\ref{eqn: euler mukai}), the Mukai vector given by (\ref{eqn: mukai vector}) is a map
 \begin{equation}\label{eq: mukai map}
 v: (K(\cD), -\chi(-, -)) \longrightarrow (H^*(X,\bZ),\langle-, -\rangle).
 \end{equation}
 The Mukai lattice is an extension of the K3 lattice on $H^2(X,\bZ)$ so that $H^0(X,\bZ)\oplus H^4(X,\bZ)$ isometric to a hyperbolic plane.  Hence, the Mukai lattice is isometric to
\begin{equation}
\Lambda_{\Muk}=U\oplus\Lambda= U^{\oplus 4}\oplus (-E_8)^{\oplus 2},
\end{equation}
which has signature $(4, 20)$.  

The numerical Grothendieck group is defined to be the quotient 
\begin{equation}
N(\cD)=K(\cD)/\ker \chi(-, -).
\end{equation}
By this definition, the Mukai pairing on $N(\cD)$ is nondegenerate. {When $\dim_\bC(X)=2$, it is known that the Chern character map, and so the Mukai vector map $v$, descends to a map on $N(\cD)$, and this map is injective.}     So it identifies $N(\cD)$ with its image 
\begin{equation}\label{eq: N mukai identification}
(N(\cD), -\chi(-,-))  \cong \left(H^0(X,\bZ)\oplus \NS(X) \oplus H^4(X, \bZ),\  \l-, -\r \right)\subset (H^*(X,\bZ), \langle-,-\rangle). 
\end{equation}
The numerical Grothendieck group with the Mukai pairing is a lattice with signature $(2, \rho(X))$, where $\rho(X)=\rk(\NS(X))$ is the Picard rank.

\subsubsection{Stability conditions} \label{sec: stability background} The space of locally finite numerical Bridgeland stability conditions, $\Stab(\cD)$, on $\cD:=\cD^b \Coh(X)$ consists of all pairs $\sigma=(Z, \cP)$, where 
 \begin{itemize}
 \item $Z: N(\cD)\to\bC$ is a group homomorphism called the central charge, and 
 \item $\cP:=\{\cP(\phi)\}_{\phi\in \bR}$ is a collection of full additive subcategories of $\cD$, one for each $\phi\in \bR$ (objects in $\cP(\phi)$ are called the $\sigma$-semistable objects of phase $\phi$),
 \end{itemize}
 such that:
 \begin{enumerate}
 \item if $0\neq E\in \cP(\phi)$, then $Z(E)\in \bR_{>0}\cdot e^{i\pi\phi}$;
 \item $\cP(\phi+1)=\cP(\phi)[1]$;
 \item  if $\phi_1>\phi_2$ and $E_j\in \cP(\phi_j)$,  then $\Hom(E_1, E_2)=0$;
 \item (Harder-Narasimhan filtration) for each $0\neq E \in \cD$, there exists a collection of triangles
\[
\xymatrix@C=.4em{
0 \ar@{==}[r] & E_{0} \ar[rrrr] &&&& E_{1} \ar[rrrr] \ar[dll] &&&& E_{2}
\ar[rr] \ar[dll] && \cdots \ar[rr] && E_{k-1}
\ar[rrrr] &&&& E \ar[dll]  \\
&&& B_{1} \ar@{-->}[ull] &&&& B_{2} \ar@{-->}[ull] &&&&&&&& B_k \ar@{-->}[ull] 
}
\]
with $B_j\in \cP(\phi_j)$ and $\phi_1>\phi_2>\cdots >\phi_k$;
 \item (support property) there is a constant $C>0$ and a norm $\| \cdot \|$ on $N(\cD)\otimes_\bZ \bR$ such that for any semistable object $E$, we have $\| E\| \leq C|Z(E)|$.
  \end{enumerate}
 
 {Note that the support property implies the set $\{Z(E)\}$ of central charges  of $\sigma$-semistable objects are discrete in $\bC$ and has no accumulation point. Indeed, otherwise there would be a sequence of classes $[E_n]\in N(\cD)$ supporting $\sigma$-semistable objects such that $\lim_{n\rightarrow\infty}|Z(E_n)|$ is finite. Together with $\limsup_{n\rightarrow\infty}\|[E_n]\|=\infty$, this contradicts with the support property.
 }

There is a left action on $\Stab(\cD)$ by the group of triangulated autoequivalences $\Aut(\cD)$ of $\cD$, and there is a right action on $\Stab(\cD)$ by $\widetilde{\GL^+(2;\bR)}$ (the universal cover of $\GL^+(2,\bR)$)  which descends to a $\GL^+(2,\bR)$ action on $Z : N(\cD) \to (\bC \cong \bR^2)$ by post-composition.  {Note that by the Gram-Schmidt procedure, $\GL^+(2,\bR)$ deformation retracts onto $SO(2,\bR)\cong S^1$, so $\pi_1\left(\GL^+(2,\bR)\right)\cong\bZ$, with the generator of the fundamental group acting by $e^{2\pi i}$.}

The non-degeneracy of the Mukai pairing gives an identification 
\begin{equation}\label{eq: central charge identification}
\Hom(N(\cD), \bC)\cong N(\cD)_\bC:=N(\cD) \otimes \bC, \quad Z\mapsto \varphi, \quad Z(v)=\l \varphi, v\r.   
\end{equation}
Bridgeland \cite{BridgelandAnnals} showed that $\Stab(\cD)$ has the structure of a complex manifold, and the forgetful map,
\begin{equation}\label{eq: forgetful map}
\pi: \Stab(\cD) \to \Hom(N(\cD), \bC)\cong N(\cD)_\bC, \quad \sigma=(Z,\cP) \mapsto Z,
\end{equation}
gives a local isomorphism.

Of special interest is the subset $U(\cD)\subset \Stab(X)$ of geometric stability conditions. 
\begin{defn}\label{def: geometric stab}
A \emph{geometric stability condition} is a $\sigma\in \Stab(X)$ for which all skyscraper sheaves, $\cO_x$ where $x\in X$, are $\sigma$-stable and of the same phase. 
\end{defn}
\noindent {Note that $U(\cD)$ is connected, due to} \cite[Proposition 11.2 and the discussion after Corollary 11.3]{BridgelandK3}.  Let $\Stab^\dagger(\cD)$ be the connected component of $\Stab(\cD)$ that contains $U(\cD)$. Bridgeland \cite{BridgelandK3} showed that there is some autoequivalence of $\cD$ that maps $\Stab^\dagger(\cD)$ into the closure of $U(\cD)\subset \Stab^\dagger(\cD)$. 

Furthermore,  \cite{BridgelandK3} showed that for any $\sigma\in U(\cD)$, there is a unique $\tilde g\in \widetilde{\GL^+(2,\bR)}$ such that $\sigma\cdot \tilde g\in V(\cD)$, where 
\begin{equation}\label{eq: def V}
V(\cD)=\{\sigma\in U(\cD)\mid \pi(\sigma)\in\cQ(X),\  \cO_x \text{ is $\sigma$-stable of phase 1 } \forall x\in X\},
\end{equation}
where $\cQ(X)\subset N(\cD)\otimes \bC$ is the set of vectors whose real and imaginary part spans positive 2-planes in $N(\cD)\otimes \bR$.
{It is worth mentioning that $\cQ(X)$ is closely related to the period domain associated to $N(\cD)$}, and can be identified with 
\begin{equation} 
\cQ(X) = \{\varphi=\exp(B+i\omega) \ | \ B+i\omega\in \NS\otimes \bC, \ \omega^2>0\}\subset N(\cD)_\bC.
\end{equation}
Note that 
\begin{equation}
\varphi=\exp(B+i\omega)=\left(1, B, \frac{B^2-\omega^2}{2}\right)+i\Big(0, \omega, B.\omega\Big).
\end{equation}

Bridgeland \cite{BridgelandK3} also showed that the forgetful map $\pi$ is injective when restricted to $V(\cD)$.  The image of 
\begin{equation}
\pi |_{V(\cD)}: V(\cD) \stackrel{\sim}{\longrightarrow} L(\cD)
\end{equation}
is
\begin{equation}
L(\cD)=\{\varphi\in \cK(X)\ |\ \text{{$\l\varphi, \delta\r \notin \bR_{\leq 0}$}} \text{ for all } \delta\in \Delta^+(N(\cD)) \}\subset \cK(X),
\end{equation}
where
\begin{equation}
\cK(X):=\{\varphi=\exp(B+i\omega)\in \cQ(X) \ | \ \omega\in \Amp(X)\}\subset \cQ(X)
\end{equation}
and
\begin{equation}
\Delta^+(N(\cD))=\{\delta=(r,D,s) \in N(\cD) \ |\ \l \delta, \delta\r=-2, \ r>0\}.
\end{equation}
So the forgetful map identifies stability conditions $\sigma=(Z, \cP)\in V(\cD)$ with its image $\varphi=\exp(B+i\omega)\in L(\cD)$ for some $B,\omega \in \NS(X)_\bR$ with $\omega\in \Amp(X)$.  By  (\ref{eq: central charge identification}), this $\varphi=\exp(B+i\omega)$ is identified with the central charge homomorphism
 \begin{equation}\label{eq: Z for V} 
 Z(v)= \l \re(\varphi), v \r+i\l \im(\varphi), v\r 
   = \left\l \left(1, B, \frac{B^2-\omega^2}{2}\right), v\right\r + i\Big\l (0, \omega, B.\omega), v\Big\r.
 \end{equation}

\subsection{Special Lagrangian classes} \label{sec: sLag background} Given a K\"ahler class in $H^{1,1}(X,\bR)$, it is uniquely represented by a Ricci-flat K\"ahler form $\omega$.  Below, we use the same notation $\omega$ to denote either the K\"ahler form or the K\"ahler class, depending on the context. For a Lagrangian submanifold $L$ in $(X,\omega)$, let $[L]^{Pd}\in H^2(X,\bZ)$ be the Poincare dual of the homology class of $L$, then we have the intersection pairing $[L]^{Pd}.\omega=\int_{L}\omega|_L=0$.
Conversely, by (\cite{SW01}, Corollary 2.4), if $\gamma\in H^2(X,\bZ)$ such that $\gamma.\omega=0$, then it represents an immersed Lagrangian submanifold.  Hence $\gamma\in H^2(X,\bZ)$ represents an immersed Lagrangian submanifold if and only if $\gamma$ belongs to
\begin{equation}\label{eq: Lag def}
\Lag(X,\omega):=H^2(X,\bZ)\cap \omega^{\perp} \subset H^2(X,\bZ). 
\end{equation}
We call $\Lag(X,\omega)$, with the intersection pairing, the Lagrangian class lattice. 

In addition to the K\"ahler form $\omega$, if we also take a nonvanishing holomorphic 2-form $\Omega$ on $X$ into account, then we can  define special Lagrangian (sLag) submanifolds of phase $\phi$, which are Lagrangian submanifolds $L$ such that $\iota^*\Omega=e^{i\phi}d\vol|_L$, where $\phi$ is a constant and $d\vol$ is the volume form induced by the K\"ahler metric determined by $\omega$.  Define the set of irreducible special Lagrangian classes with respect to $\omega, \Omega$ to be
\begin{equation}
    \SLag(X,\omega, \Omega)=\{\gamma\in H^2(X,\bZ): \exists\  \text{irreducible } \sLag 
    \ L \text{ with } [L]^{Pd}=\gamma\}.
\end{equation}
{Using the hyperk\"ahler rotation and standard facts about holomorphic curves, one has the following lemma.}
\begin{lemma}[Lai-Lin-Schaffler, {\cite[Lemma 2.2]{LaiLinSchaffler}}] \label{lemma: sLag}
$$\SLag(X,\omega, \Omega)\subset\{\gamma\in \Lag(X,\omega) : \gamma^2\geq -2\}.$$
\end{lemma}
\noindent Note that notation-wise for us $\SLag(X,\omega, \Omega)$ consists of only the classes represented by special Lagrangians that are irreducible.
The same notation in \cite{LaiLinSchaffler} does not require irreducibility, but the above lemma is indeed about irreducible Lagrangians.

\begin{rmk}\label{rmk: Lag signature} By \cite[Proposition~3.3]{LaiLinSchaffler}, given a lattice $\Lambda'\subsetneq\Lambda\cong H^2(X,\bZ)$, there exists a K\"ahler K3 surface $(X,\omega)$ such that $\Lag(X,\omega)\cong \Lambda'$ if and only if $\Lambda'$ is proper, saturated, and such that $(\Lambda')^\perp\subseteq\Lambda$ contains a vector with positive self-intersection. Hence in general we don't have control on the signature of $\Lag(X,\omega)$; also, the lattice could be degenerate. 
\end{rmk}

\begin{rmk}\label{rmk: hms} In this remark, we briefly describe the numerical Grothendieck group of the derived Fukaya category of $X$, which is expected to correspond to $N(\cD(Y))$ if $X$ and $Y$ are mirror K3s.  See \cite{SS20} for a more detailed exposition. Denote by $\Fuk(X)$ the Fukaya category of $X$, whose objects are unobstructed immersed Lagrangian submanifolds.   It is  an $A_\infty$-category that is linear over the Novikov field $\mathbb F$ over $\bC$ given by $
\mathbb F:= \big\{\sum\limits_{j=0}^{\infty}a_j q^{\lambda_j} \ | \ a_j\in \bC, \lambda_j\in \bR,\ \lim\limits_{j\to \infty} \lambda_j=+\infty\big\}$.  Denote by $\cF:=\cF(X)$ the triangulated category given by the split-closures of the category of twisted complexes on the Fukaya category. Let $K(\cF)$ be the Grothendieck group, then there is a composition of the Chern character map and the open-closed map 
\begin{equation} \label{eqn: ch oc}
K(\cF) \stackrel{\ch}{\longrightarrow} HH_0(\cF) \stackrel{ OC}{\longrightarrow} H^2(X, \mathbb F).
\end{equation}
Shklyanov \cite{Shk13} introduced Mukai pairing for dg categories and proved that the Mukai pairing on $HH_0(\cF)$ satisfies
\begin{equation}
\l \ch(E), \ch(F)\r = -\chi(E,F).
\end{equation}
{An object in $K(\cF)$ needs not be  geometric}, but when it is, i.e. when it can be represented by a unobstructed immersed Lagrangian submanifold, \cite[Lemma 5.13]{SS20} shows that (\ref{eqn: ch oc}) is the map $[L] \mapsto [L]^{Pd} \in \Lag(X,\omega)\subset H^2(X,\bZ)$. When $L_1, L_2$ are immersed Lagrangian submanifolds, their Euler pairing is the Euler characteristics of the Floer cohomology, and we see that the Mukai pairing coincides with the intersection pairing on $H^2(X,\bZ)$
\begin{equation}
\begin{array}{ll}
\l \ch(L_1), \ch(L_2) \r = - \chi(L_1, L_2) & =- \chi(HF^*(L_1, L_2)) =  [L_1].[L_2].
\end{array}
\end{equation}
Let $N(\cF):= K(\cF)/\ker \chi(-,-)$ be the numerical Grothendieck group, then the map $K(\cF)\to H^2(X,\mathbb F)$ descends to an injective map $N(\cF)\to H^2(X,\mathbb F).$ If $X$ and $Y$ are mirror K3 surfaces, then the expectation is that 
\begin{equation}
N(\cF(X))=N(\cD(Y)),
\end{equation}
so then $N(\cF(X))$ would have signature $(2,\rho)$ where $\rho=\rk(\Pic(Y))$.  It is not known whether $N(\cF(X))$  can be identified with $\Lag(X,\omega)$, 
{see} \cite[Remark~7.6]{SS20} {for more detailed discussions}.
\end{rmk}

\section{Counting semistable Mukai vectors}\label{sec: count semistable objects}
In this section we explain Theorem \ref{thm: count stable objects}.  Let $X$ be a complex algebraic $K3$ surface,  denote by $\cD:=D^b\Coh(X)$ be the bounded derived category of coherent sheaves on $X$, and denote by $N(\cD)$ the numerical Grothendieck group, which we view as the signature $(2,\rho(X))$ Mukai lattice $H^*(X,\bZ)$ via the identification given by Equation \eqref{eq: N mukai identification}.  To proceed, please refer to the notations introduced in Section \ref{sec: stability background}. Fix a Bridgeland stability condition  $\sigma\in \Stab^\dagger(\cD)$ with central charge  {$Z_\sigma\colon N(\cD)\ra \bC$.} We are interested in the following counting function of semistable classes with bounded charges 
\begin{equation}
N_\sigma(R)=\#\{v \in N(\cD) \colon|Z_\sigma(v)|\leq R\text{ and }
\exists\ \sigma\text{-semistable object }E\text{ with }[E]=v\},
\end{equation}
where $[E]$ is the image of $E$ under the Mukai map in Equation \eqref{eq: mukai map}.
Because by \cite{BridgelandK3},  there is some autoequivalence of $\cD$ that maps $\Stab^\dagger(\cD)$ into the closure of $U(\cD)$, we will just do the count for geometric stability conditions.  As mentioned in Section \ref{sec: stability background}, for any stability condition $\sigma\in U(\cD)$, there is a unique $\tilde g\in \widetilde{\GL^+(2,\bR)}$ such that $\sigma\cdot \tilde g\in V(\cD)$, which corresponds to $\pi(\sigma\cdot \tilde g)\in L(\cD)$ via the isomorphism $\pi|_{V(\cD)}$. So for each element $\pi(\sigma)\in \pi(U(\cD))$, there is a unique $g=\pi(\tilde g)\in \GL^+(2,\bR)$ such that $\pi(\sigma)\cdot g\in L(\cD)$.  So in this section, we will first look into $N_\sigma(R)$ for $\sigma\in V(\cD)$ and then consider other stability conditions in $ U(\cD)$ obtained via $\widetilde{\GL^+(2,\bR)}$ action.

Given a Mukai vector $v$, by \cite{BridgelandK3} there is a locally finite set of walls (which are real codimension one submanifolds) in $\Stab^\dagger(\cD)$ such that so long as $\sigma$ varies within a chamber, the set of $\sigma$-semistable objects with class $v$ does not change.  Following \cite{BayerMacri}, we make the following definition.
\begin{defn}\label{def: generic}
A stability condition is \emph{generic} with respect to $v$ if it does not lie on a wall. 
\end{defn} 
\noindent Almost every point $\sigma\in \Stab^\dagger(\cD)$ is generic with respect to all Mukai vectors.  Indeed, each Mukai vector gives a locally finite set of walls, and there are countably many Mukai vectors, so locally the union of all walls is a measure zero set in $\Stab^\dagger(\cD)$. The following useful theorem characterizes Mukai vectors of semistable objects.
\begin{thm}[Bayer-Macr\`{i},  {\cite[Theorem 2.15]{BayerMacri}}] \label{thm: BayerMacri} Let $v=mv_0$ be a Mukai vector with $v_0$ primitive and $m>0$, and let $\sigma\in \Stab^\dagger(\cD)$ be a generic stability condition with respect to $v$, then $v$ supports a semistable object if and only if $v_0^2\geq -2$.  
\end{thm}
 \noindent For a $\sigma\in \Stab^\dagger(\cD)$ that is generic with respect to all Mukai vectors, Theorem \ref{thm: BayerMacri} immediately gives that 
\begin{equation}
N_\sigma(R) =\#\{v \in N(\cD) \colon  
 v= mv_0,  m\in \bZ_+, v_0 \text{ primitive, } v_0^2\geq -2, |Z_\sigma(v)|\leq R\}.
\end{equation}
Note that there is no $v\in N(\cD)$ with $v^2=-1$ because the Mukai lattice is even, so
\begin{multline}\label{eq: N sphere}
N_\sigma(R)=\left(\#\{v\in N(\cD) \colon v^2\geq 0, |Z_\sigma(v)|\leq R\}\right)\\ 
 +\left(\#\{v\in N(\cD):  v=mv_0, m\in \bZ_+, v_0^2=-2, |Z_\sigma(v)|\leq R \}\right).
\end{multline}
In the equation above, we didn't explicitly specify that $v_0$ is primitive, but it is, since any $v_0$ such that $v_0^2=-2$ is primitive.  

Let us first discuss the second term in Equation \eqref{eq: N sphere}.  There is actually an upper bound on how large $m$ can be.  Due to the support {property in the} definition of locally finite numerical Bridgeland stability conditions introduced in Section \ref{sec: stability background}, the systole introduced in \cite{FanSystole}, which is defined as 
\begin{equation}
    \mathrm{sys}(\sigma):=\min\{|Z_\sigma(v(E))|:  E \text{ is a $\sigma$-semistable  object}\},
\end{equation}
is attained and is a positive number, {see} \cite[Remark~2.5]{FanSystole}.
Consequently, if $v=m v_0$ and $|Z_\sigma(v)|<R$, then we must have $m<\frac{R}{\mathrm{sys}(\sigma)}$ since $|Z_\sigma(v)|=m|Z_\sigma(v_0)|$ and  $|Z_\sigma(v_0)|\geq \mathrm{sys}(\sigma)$. So the second term in Equation \eqref{eq: N sphere} is bounded by
\begin{equation}\label{eq: sphere bound}
\begin{split}
&\#\{v\in N(\cD):  v=mv_0, m\in \bZ_+, v_0^2=-2, |Z_\sigma(v)|\leq R \}\\
<&  \frac{R}{sys(\sigma)}\cdot \#\{v\in N(\cD): v^2=-2, |Z_\sigma(v)|\leq R\}\\
=& o(R^{\rho+2})
\end{split}.
\end{equation}
The last line of the above equation is due to the fact that  $\#\{v\in N(\cD): v^2=-2, |Z_\sigma(v)|\leq R\}$ is of order $R^\rho$ as explained in Lemma \ref{lem: -2 class count} below.  So when multiplied with $\frac{R}{\mathrm{sys}(\sigma)}$, overall we have $o(R^{\rho+2})$. 
\begin{lemma}\label{lem: -2 class count}
{The count $\#\{v\in N(\cD): v^2=-2, |Z_\sigma(v)|\leq R\}$ is of order $R^\rho$. }
\end{lemma}
\begin{proof}
This is a consequence of the main theorem of \cite[Theorem 1.2]{DukeRudnickSarnak}, and let us explain that we are indeed in the correct setting to this theorem.  For $v=(r, D=(D_1,\ldots, D_\rho ), s)\in H^0(X,\bZ)\oplus \NS(X)\oplus H^4(X,\bZ)$,
we have  $F(v)=v^2=D^2-2rs$ is an integral polynomial.  So we are looking at the asymptotics as $R\to \infty$ of the number of integer points in the affine homogeneous variety $V=\{x\in \bC^{\rho+2}: F(x)=-2\}$ and such that the positive semidefinite norm given by  $|Z_\sigma(x)|^2$ is less than $R^2$. Note that $F$ is a quadratic form of signature $(2,\rho)$. (This is exactly the setting of \cite[Example 1.5]{DukeRudnickSarnak}, which start out with $F$ being a quadratic form of an arbitrary signature and then go on to analyze a more specialized case where the signature of $F$ is different from ours.)  
The variety $V$ is a symmetric space of the form $\mathrm{SO}(2,\rho-1)\backslash \mathrm{SO}(2,\rho)$. Denote by $H=\mathrm{SO}(2,\rho-1)$.  By the main theorem of \cite{DukeRudnickSarnak}, if $\vol(H(\bZ)\backslash H(\bR))<\infty$, then $\#\{v\in V(\bZ): |Z_\sigma(v)|^2<R^2\}$ is of order $R^\rho$; otherwise, it is of order $R^\rho\log R$. So, in either case, this count is of order $o(R^{\rho+1})$.  For our $H$, we in fact have $\vol(H(\bZ)\backslash H(\bR))<\infty$ due to the fundamental theorem of Borel and Harish-Chandra. 
\end{proof}

Below we will discuss the first term in Equation \eqref{eq: N sphere}.  We will use Lemma \ref{lem: Gauss circle} below, which is a classical result.  

\begin{lemma}\label{lem: Gauss circle} {(Gauss Circle Problem.)} Let $\Upsilon(1)$ be a closed subset of $\mathbb R^n$ with a piecewise smooth boundary of measure zero.  Let $\Upsilon(R)=\{Rx \mid x\in \Upsilon(1)\}$ be the $R$-dilate of $\Upsilon(1)$.  Let $N_\Upsilon(R)=\#\{\Upsilon(R)\cap \bZ^n\}$ be the number of lattice points inside $\Upsilon(R)$.  Then 
\[
\lim_{R\to \infty} \frac{N_\Upsilon(R)}{R^n}=\Vol (\Upsilon(1)).
\]
\end{lemma}

We see that the first term in Equantion \eqref{eq: N sphere} is the count of the number of lattice points in the set 
\[
\Upsilon_\sigma(R)=\{v\in N(\cD)_\bR\cong \bR^{\rho+2}, v^2\geq 0, |Z_\sigma(v)|^2\leq R^2\}.
\]
Because $|Z_\sigma(v)|^2$ is a positive semidefinite quadratic form, $\Upsilon_{\sigma}(1)$ is a closed set with piecewise smooth boundary of measure zero.  {Using Lemma} \ref{lem: Gauss circle} {with $\Upsilon=\Upsilon_\sigma$ and $n=\rho+2$}, we get that $\lim_{R\to \infty} \frac{N_{\Upsilon_\sigma}(R)}{R^{\rho+2}}= \Vol(\Upsilon_{\sigma}(1))$.

Putting the two terms of Equation \eqref{eq: N sphere} together, we have  
\begin{equation}
    N_\sigma(R)=C(\sigma)R^{\rho+2}+o(R^{\rho+2}),
\end{equation}
where 
\begin{equation}\label{eq: Csigma}
      C(\sigma)=\Vol(\Upsilon_\sigma(1))=\Vol\{v\in N(\cD)_\bR: v^2\geq 0, |Z_\sigma(v)|^2\leq 1\}.\\
\end{equation}

\begin{center}
    \fbox{\textbf{Geometric stability conditions of phase 1}}
\end{center}

For $\sigma \in V(\cD)$ a geometric stability condition of phase 1, $Z_\sigma(v)$ is given by Equation \eqref{eq: Z for V}, so according to Equation \eqref{eq: Csigma}, the leading term of $N_\sigma(R)$ is equal to $C(\sigma)R^{\rho+2}$, where
\begin{equation}\label{eq: Csigma phase 1}
C(\sigma)=\Vol\left(\left\{v\in N(\cD)_\bR\colon v^2\geq0,\ \left<\re(\varphi),v\right>^2+\left<\im(\varphi),v\right>^2\leq 1 \right\}\right)
\end{equation}
and
\begin{equation}
\varphi=\exp(B+i\omega)= \left(1, B, \frac{B^2-\omega^2}{2}\right)+ i(0, \omega, B.\omega)\in N(\cD)_\bC.
 \end{equation}
We certainly have $\re(\varphi),\im(\varphi)\in N(\cD)_\bR$, so we can write any $v\in N(\cD)_\bR$ as
\begin{equation}
v=\alpha_1\re(\varphi)+\alpha_2\im(\varphi)+\beta,
\end{equation}
where $\alpha_1,\alpha_2\in\bR$ and $\beta\in\left<\re(\varphi),\im(\varphi)\right>^{\perp_{N(\cD)_\bR}}$. 

Recall that $\{\re(\varphi),\im(\varphi)\}\subseteq N(\cD)_\bR$ forms an orthogonal basis of a positive definite $2$-plane, hence $\left<\re(\varphi),\im(\varphi)\right>^{\perp_{N(\cD)_\bR}}$ is of signature $(0,\rho)$.
Let $\{w_1,\ldots,w_\rho\}$ be a basis of $\left<\re(\varphi),\im(\varphi)\right>^{\perp_{N(\cD)_\bR}}$ such that $\left<w_i,w_j\right>=-\delta_{ij}$.
Then we can write $\beta$ as
\begin{equation}\label{eq: w basis}
\beta=\beta_1w_1+\cdots+\beta_\rho w_\rho \text{ \ for \ } \beta_j\in\bR.
\end{equation}
Then
\begin{equation}\label{eq: alpha beta 1}
v^2\geq0\Longleftrightarrow\omega^2(\alpha_1^2+\alpha_2^2)\geq\beta_1^2+\cdots+\beta_\rho^2,
\end{equation}
and
\begin{equation}\label{eq: alpha beta 2}
\left<\re(\varphi),v\right>^2+\left<\im(\varphi),v\right>^2\leq1\Longleftrightarrow\frac{1}{(\omega^2)^2}\geq\alpha_1^2+\alpha_2^2.
\end{equation}

To compute the volume of $\{\alpha_1,\alpha_2,\beta_1,\ldots,\beta_\rho\}$ satisfying Equations \eqref{eq: alpha beta 1} and \eqref{eq: alpha beta 2}, we introduce a coordinate $r$ such that  $\alpha_1^2+\alpha_2^2=r^2$, then
\begin{equation}
0\leq r\leq\frac{1}{\omega^2} \text{ \ and \ } \beta_1^2+\cdots+\beta_\rho^2\leq\omega^2r^2.
\end{equation}
Hence the volume with respect to $\{\alpha_1,\alpha_2,\beta_1,\ldots,\beta_\rho\}$ is:
\begin{equation}
\begin{split}
&2\pi\int_0^{\frac{1}{\omega^2}}r\cdot\Vol B_\rho(r\sqrt{\omega^2})dr\\
=& 2\pi\int_0^{\frac{1}{\omega^2}}r\cdot\frac{\pi^{\rho/2}(\omega^2)^{\rho/2}r^\rho}{\Gamma(\frac\rho2+1)}dr\\
=&\frac{2\pi^{(\rho+2)/2}}{(\rho+2)\Gamma(\frac\rho2+1)(\omega^2)^{(\rho+4)/2}}.
\end{split}
\end{equation}

In order to get the volume in Equation \eqref{eq: Csigma phase 1}, we need to analyze the  effect of the change-of-basis from the standard basis $\be=(e_1,\ldots, e_{\rho+2})$ of  $H^0(X,\bZ)\oplus NS(X)\oplus H^4(X,\bZ)$ to the basis $\bf=(f_1,\ldots, f_{\rho+2})$ where $f_1=\re(\varphi)$, $f_2=\im(\varphi)$, and $f_j=w_{j-2}$ for $j=3,\ldots, \rho+2$. Denote by $A$ the change of basis matrix, i.e. $A\be=\bf$.
Then 
\[
\langle \bf_i, \bf_j\rangle=\langle \sum_k a_{ik}e_k,\sum_\ell a_{j\ell}e_\ell\rangle= \sum_{k,\ell} a_{ik}a_{j\ell}\langle e_k, e_\ell\rangle,\text{ i.e. } A[\langle e_i, e_j\rangle]A^T=[\langle f_i, f_j\rangle],
\]
where $[\langle e_i, e_j\rangle]$ and $[\langle f_i, f_j\rangle]$ denote the Gram matrices for $\be$ and $\bf$, respectively.  So then 
\[
(\det A)^2=\frac{\det [\langle f_i, f_j\rangle]}{\det [\langle e_i, e_j\rangle]}.
\]
Since $\det[\langle f_i, f_j\rangle]=(\omega^2)^2$ and $\det [\langle e_i, e_j\rangle]=\disc \NS(X)$, we get that the determinant of the change of basis matrix is 
\begin{equation}
\det (A)=\frac{\omega^2}{\sqrt{|\disc(\NS(X))|}}.
\end{equation} Hence the final result is
\begin{equation}\label{eq: Csigma phase 1 result}
C(\sigma)=\frac{2\pi^{(\rho+2)/2}}{(\rho+2)\Gamma(\frac\rho2+1)(\omega^2)^{(\rho+2)/2}\sqrt{|\disc\NS(X)|}}.
\end{equation}
We see that the leading coefficient depends on the following geometric data 
\begin{itemize}
    \item rank and discriminant of $\NS(X)=\left<\Omega\right>^{\perp_{H^2(X,\bZ)}}$, and
    \item $\omega^2$ (the ``B-field" $B$ does not matter).\\
\end{itemize}

\begin{center}
\fbox{\textbf{Geometric stability conditions of any phase}}
\end{center}

For a geometric stability condition $\sigma\in U(\cD)$ of phase other than 1, we know there is a unique $g\in \GL^+(2,\bR)$ such that $Z_\sigma\cdot g=Z_{\sigma'}\in L(\cD)$ for some $\sigma'\in V(\cD)$. Note that $\GL^+(2,\bR)\cong \bR^+\times \SL(2,\bR)$.  If $g\in \bR^+$, it is simply a scaling $Z_{\sigma'}=gZ_\sigma$, so $N_{\sigma'}(R)=N_{\sigma}(R/g)$, and so $C(\sigma')=C(\sigma)/g^{\rho+2}$.  If $g$ belongs to the rotation part of $\SL(2,\bR)$ doesn't change the norm $|Z_{\sigma'}|=|Z_{\sigma}|$, so $N_{\sigma'}(R)=N_{\sigma}(R)$. So the only thing left to consider a shear by $\kappa+i\lambda$ with $\lambda>0$, which is represented by the matrix $\begin{pmatrix}1 & \kappa \\ 0 & \lambda \end{pmatrix}$ so $Z_{\sigma'}=\re Z_{\sigma}+i(\kappa \re Z_{\sigma}+\lambda\im Z_{\sigma})$.  Then Equation \eqref{eq: Csigma} in this context is
\begin{equation}\label{eq: Csigma any phase}
 C(\sigma)= \Vol\{v\in N(\cD)_\bR, v^2\geq 0, \l \re(\varphi), v\r^2+\left(\l \re(\varphi),v\r \kappa+ \l\im(\varphi), v\r \lambda\right)^2\leq 1\}.
 \end{equation}

 The computation is very similar to what we did above in the phase 1 case.  In exactly the same way, we write  $v=\alpha_1\re(\varphi)+\alpha_2\im(\varphi)+\beta,$ where $\beta=\beta_1w_1+\cdots +\beta_\rho w_\rho$ as in Equation \eqref{eq: w basis}. So the determinant of the change of basis matrix for $\{\re{\varphi}, \im{\varphi}, w_1,\ldots, w_\rho\}$  is again  $\frac{\omega^2}{\sqrt{|\disc(\NS(X))|}}$.
 Then the $v^2\geq 0$ condition is the same as in Equation \eqref{eq: alpha beta 1}, and Equation \eqref{eq: alpha beta 2} in this context is
 \begin{equation} \label{eq: alpha beta 2 shear}
  \l \re(\varphi), v\r^2+\left(\l \re(\varphi),v\r \kappa+ \l\im(\varphi), v\r \lambda\right)^2\leq 1 \quad \Leftrightarrow \quad \alpha_1^2+(\kappa\alpha_1+\lambda\alpha_2)^2\geq \left(\frac{1}{\omega^2}\right)^2.
 \end{equation}
Again, we can  compute $C(\sigma)$ given in Equation \eqref{eq: Csigma any phase} by computing the volume of $\{\alpha_1,\alpha_2,\beta_1,\ldots,\beta_\rho\}$ satisfying Equations \eqref{eq: alpha beta 1} and \eqref{eq: alpha beta 2 shear}, and then multiply by the determinant of the change of basis matrix for $\{\re(\varphi),\im(\varphi), w_1,\ldots w_\rho\}$.  To compute the volume of $\{\alpha_1, \alpha_2,\beta_1,\ldots, \beta_\rho\}$ satisfying the two conditions, we change to polar coordinates by letting
  \begin{equation}\label{eq: alpha to polar}
  \alpha_1=r\cos \theta, \quad \kappa\alpha_1+\lambda\alpha_2=r\sin\theta.
  \end{equation}
  Then $0\leq r\leq\frac{1}{\omega^2}$ and 
  \begin{equation}
  \beta_1^2+\cdots +\beta_\rho^2\leq  \omega^2(\alpha_1^2+\alpha_2^2)=r^2\omega^2\left(\cos^2\theta+ \frac{1}{\lambda^2} (\sin\theta-\kappa\cos\theta)^2\right).
  \end{equation}
From Equation \eqref{eq: alpha to polar}, we see that the Jacobian determinant for the change of variable between $(\alpha_1,\alpha_2)$ and $(r,\theta)$ is
\begin{equation}
    \left|\frac{\partial(\alpha_1,\alpha_2)}{\partial(r,\theta)} \right|=\det \begin{bmatrix}\cos\theta & -r\sin \theta\\ \frac{1}{\lambda}(\sin\theta-\kappa \cos \theta)& \frac{r}{\lambda}(\cos \theta+\kappa\sin\theta) \end{bmatrix}=\frac{r}{\lambda}.
\end{equation}
So the volume with respect to $\{\alpha_1,\alpha_2,\beta_1,\ldots, \beta_\rho\}$ is:
 \begin{equation}
 \begin{split}
 &\displaystyle\int_0^{\frac{1}{\omega^2}} \displaystyle\int_0^{2\pi} \left|\frac{\partial(\alpha_1,\alpha_2)}{\partial(r,\theta)} \right| \Vol B_\rho\left(r\sqrt{\omega^2\left(\cos^2\theta+ \frac{1}{\lambda^2} (\sin\theta-\kappa\cos\theta)^2\right)}\right) d\theta dr\\
 =&  \displaystyle\int_0^{\frac{1}{\omega^2}} \displaystyle\int_0^{2\pi}  \frac{r}{\lambda}  \frac{\pi^{\rho/2}(\omega^2)^{\rho/2} r^\rho}{\Gamma\left(\frac{\rho}{2}+1\right)} \left(\cos^2\theta+ \frac{1}{\lambda^2} (\sin\theta-\kappa\cos\theta)^2\right)^{\rho/2}d\theta dr\\
 =& \frac{\pi^{\rho/2}(\omega^2)^{\rho/2}}{{\lambda\Gamma\left(\frac{\rho}{2}+1\right)}} \displaystyle\int_0^{\frac{1}{\omega^2}}  r^{\rho+1}dr \displaystyle\int_0^{2\pi}  \left(\cos^2\theta+ \frac{1}{\lambda^2} (\sin\theta-\kappa\cos\theta)^2\right)^{\rho/2}d\theta\\
 =& \frac{\pi^{\rho/2}}{(\rho+2)\Gamma\left(\frac{\rho}{2}+1\right)\lambda(\omega^2)^{(\rho+4)/2}}\displaystyle\int_0^{2\pi}   \left(\cos^2\theta+ \frac{1}{\lambda^2} (\sin\theta-\kappa\cos\theta)^2\right)^{\rho/2}d\theta.
 \end{split}
 \end{equation}
Then multiplying by $\frac{\omega^2}{\sqrt{|\disc(\NS(X))|}}$ gives
\begin{equation}
    C(\sigma)=\frac{\pi^{\rho/2}\displaystyle\int_0^{2\pi}   \left(\cos^2\theta+ \frac{1}{\lambda^2} (\sin\theta-\kappa\cos\theta)^2\right)^{\rho/2}d\theta}{(\rho+2)\Gamma\left(\frac{\rho}{2}+1\right)\lambda(\omega^2)^{(\rho+2)/2}\sqrt{|\mathrm{Disc}\NS(X)|}}.
\end{equation}
One can see that for  $\kappa=0, \lambda=1$, this matches the phase 1 case.  Again we see that this leading coefficient does not depend on the $B$-field just like the phase 1 case.

\section{Counting special Lagrangian classes}
In this section, we explain Theorems \ref{thm: SL omega Omega} and  \ref{thm: SL P}.  First, we consider
\begin{equation}
SL_{\omega, \Omega}(R) =\#\left\{\gamma\in \SLag(X,\omega, \Omega):  |\gamma.\Omega|\leq R \right\}\subset  H^2(X,\bZ).
\end{equation}
By Lemma \ref{lemma: sLag},
\begin{equation}
SL_{\omega, \Omega}(R)\leq \# \left\{\gamma\in \Lag(X,\omega) : \gamma^2\geq -2, |\gamma.\Omega|\leq R \right\}.
 \end{equation}
Since $\gamma^2$ is an indeffinite quadratic form and $|\gamma.\Omega|^2$ is a positive semidefinite quadratic form, we can use  Lemma \ref{lemma: dilate} below to conclude that
\begin{equation}
SL_{\omega, \Omega}(R) =C(\omega,\Omega)R^{\rho+2}+o(R^{\rho+2}).
\end{equation}
where
\begin{equation}
C(\omega,\Omega)=\Vol\{\gamma\in \Lag(X,\omega)_\bR: \gamma^2\geq 0, (\gamma.\Omega)^2=(\gamma.\re\Omega)^2+(\gamma.\im\Omega)^2\leq 1\}.
\end{equation}
Note that even though $\re\Omega$ and $\im \Omega$ are orthogonal to $\omega$, they might not be in $\Lag(X,\omega)_\bR$.  This is because $\Lag(X,\omega)\subset H^2(X,\bZ)$ consists of integral vectors orthogonal to $\omega$, which might not generate the space that $\re\Omega$ and $\im \Omega$ live in.  If we work with the case where $(X,\omega)$ is such that $\Lag(X,\omega)$ has full rank, i.e. with signature $(2,19)$, then $\re\Omega$ and $\im \Omega$ are both in $\Lag(X,\omega)_\bR$.  In this case the calculation is exactly the same as the calculation we did in Section \ref{sec: count semistable objects} for the geometric stability conditions of phase 1, the only difference being that we replace $\rho$ in that calculation by $19$ here, and replace $\omega^2$ there by $K_\Omega=(\re\Omega)^2=(\im\Omega)^2$ here. We get 
\begin{equation}
C(\omega,\Omega)=\frac{2\pi^{21/2}}{21\Gamma(\frac{21}{2})K_{\Omega}^{{21}/2}\sqrt{|\disc \Lag(X,\omega)|}}.
\end{equation}
This case where $\Lag(X,\omega)$ having signature $(2,19)$  can be achieved when $\omega$ is rational, i.e. in $H^2(X,\bQ)$; in particular, it can be achieved for polarized K3 surfaces $(X,\omega)$.

\begin{lemma}\label{lemma: dilate}
Suppose $P$ is a positive semidefinite quadratic form and $Q$ is any quadratic form.  Then for 
\[N_{P,Q}(R)=\#\{x\in \bZ^n: P(x)\leq R^2, Q(x)\geq -c\},\]
where $c$ is a positive constant, we have 
\[
N_{P,Q}(R)=C_{P,Q}R^n+o(R^{n}),
\]
where 
\[C_{P,Q}=\Vol \{x\in \bR^n : P(x)\leq 1, Q(x)\geq 0\}.\]
\end{lemma}
\begin{proof} 
Let 
\[\Upsilon_{P,Q}(R)=\{x\in \bR^n \mid P(x)\leq R^2, Q(x)\geq 0\}.\]
{Using Lemma} \ref{lem: Gauss circle} {with $\Upsilon=\Upsilon_{P,Q}$}, we get that 
\[
\lim_{R\to \infty} \frac{\# \{x\in \bZ^n \cap \Upsilon_{P,Q}(R)\} }{R^n}= \Vol(\Upsilon_{P,Q}(1))
\]

Now consider the remaining points $\bZ^n\cap \Upsilon_{P,Q}^c(R)$, where
\[
\Upsilon_{P,Q}^c(R): = \{x\in \bR^n : P(x)\leq R^2, Q(x)\in [-c, 0)\}.
\]
Note that
\[
\Upsilon_{P,Q}^c(R) =R\cdot \widetilde\Upsilon_{P,Q}^c(R)\text{, where }  \widetilde\Upsilon_{P,Q}^c(R)=\left\{ x\in \bR^n: P(x)\leq 1, Q(x)\in \left[-\frac{c}{R^2}, 0\right]\right\}.
\] 
Whenever $R>R_0$ for some fixed $R_0$, we have
$\Upsilon_{P,Q}^c(R) \subset R \cdot \widetilde\Upsilon_{P,Q}^c(R_0)$, so 
\[
\lim_{R\to \infty} \frac{\#\{x\in \Upsilon_{P,Q}^c(R)\cap \bZ^n\}}{R^n}\leq \Vol(\widetilde\Upsilon_{P,Q}^c(R_0)).
\]
Because the boundary $\{x\in \bR^n\mid P(x)\leq 1, Q(x)=0\}$ has measure zero, the $\Vol(\widetilde\Upsilon_{P,Q}^c(R_0))$ can be made arbitrarily small by taking $R_0$ to be arbitrarily large.  So the number of remaining points $\#\{x\in \Upsilon_{P,Q}^c(R)\cap \bZ^n\} =o(R^n)$.\\
\end{proof}
 
Now let us consider the twistor plane
\begin{equation}
P:=\langle \omega, \re \Omega,  \im \Omega\rangle \subset H^2(X,\bR),
\end{equation}
where $\omega, \re\Omega, \im\Omega$ are normalized so they form an orthonormal basis of $P$ as introduced in Section \ref{sec: twistor}. The twistor plane defines a seminorm $\|\cdot\|_P$ on $H^2(X,\bZ)$, which is for $\gamma\in H^2(X,\bZ)$,
\begin{equation}
||\gamma||_P:=\sup_{u\in P, u^2=1}\gamma.u.
\end{equation}
Since $H^2(X,\bR)=P\oplus P^\perp$, for any  $\gamma\in H^2(X,\bZ)$, we can write $\gamma=\gamma_P\oplus \gamma_{P^\perp}$.  Then $\|\gamma\|_P$ is the norm of the projection $\gamma_P$ of $\gamma$ onto  the 3-plane $P$.  

{For a fixed $P$, instead of fixing an $\omega$ like we did in the count of $SL_{\omega, \Omega}(R)$, here we incorporate all $\omega_t\in \bS^2(P)$ by considering}
\begin{equation}
SL_P(R)= \#\left\{\gamma\in H^2(X,\bZ)\colon
\exists \omega_t\in \bS^2(P), \gamma\in \SLag(X, \omega_t, \Omega_t), \text{ and } |\gamma.\Omega_t|\leq R \right\}.
\end{equation}
Note that $|\gamma.\Omega_t|$ is the length of the projection of $\gamma$ to the positive 2-plane spanned by orthonormal vectors $\{\re \Omega_t, \im \Omega_t\}$.  So if there is a $t$ such that $\gamma.\omega_t=0$, i.e. $\gamma$ is a Lagrangian with respect to $\omega_t$, then \begin{equation}
    |\gamma.\Omega_t|=\|\gamma\|_P.
\end{equation}
The characterization given by Lemma \ref{lemma: sLag} for $\gamma\in \SLag(X,\omega_t,\Omega_t)$ implies that  
\begin{equation}\label{eq: SL P bound}
\begin{split}
SL_P(R)
&\leq  \#\left\{\gamma\in H^2(X,\bZ): \gamma^2\geq -2, \exists \omega_t\in \bS^2(P) \text{ s.t. } \gamma.\omega_t=0, \text{ and }  ||\gamma||_P\leq R\right\}\\
&= \#\{\gamma\in H^2(X,\bZ) : \gamma^2\geq -2, ||\gamma||_P\leq R\}.
\end{split}
\end{equation}
The last equality is due to the fact that for any $\gamma\in H^2(X,\bZ)$, it is always in the Lagrangian class for some $\omega_t$, i.e. $\gamma.\omega_t=0$ since we can always project $\gamma$ onto $P$ and then choose $\omega_t$ to be orthogonal to that.

We can then use Lemma \ref{lemma: dilate}, since $\|\gamma\|_P^2$ is positive semidefinite, to conclude that 
 the leading term of $\#\{\gamma\in H^2(X,\bZ): \gamma^2\geq -2, \|\gamma\|^2_P\leq R^2\}$ is given by the leading term of the counting function $\#\{\gamma\in H^2(X,\bZ): \gamma^2\geq 0, \|\gamma\|^2_P\leq R^2\}$.  Then using Lemma \ref{lemma: indep of P} below, and combining with Equation \eqref{eq: SL P bound}, we can conclude that 
 \begin{equation}
     SL_P(R)\leq C\cdot R^{22}+o(R^{22}),
 \end{equation}
 where $C$ is a constant independent of the choice of $P\subset \Lambda_\bR$.
 
\begin{lemma}\label{lemma: indep of P} Let $Q$ be a signature $(m, n)$ quadratic form on $\mathbb R^d$, with $d= m+n$. Given $v\in \bR^d$ and an $m$-dimensional subspace $P\subset \bR^d$ that is positive definite (i.e. for all nonzero $u\in P$, we have $Q(u,u)>0)$), define   
$$\|v\|_P = \max\{ Q(v, u): u \in P, Q(u, u) = 1\}.$$ 
Then the asymptotics of the counting function 
$$N_P(R) = \#\{ v \in \mathbb Z^d: Q(v,v)\geq 0, \|v\|_P \le R\}$$
is given by 
$$N_P(R) =c_P R^d+o(R^d),$$ 
where the constant
$$c_P = \Vol\{ v \in \mathbb R^d: Q(v,v)\geq 0, \|v\|_P \le 1\}$$ 
 is independent of the choice of $m$-dimensional positive definite subspace $P$.
\end{lemma}
\begin{proof} Because $P$ is positive definite, the set
\[
\Upsilon_P(R):=\{v\in \bR^d: Q(v,v)\geq 0, \|v\|_P\leq R\}=\{Rv: v\in \Upsilon_P(1)\}=R\cdot\Upsilon_P(1)
\]
is a {$R$-dilate} of  $\Upsilon_P(1)$.  {Using Lemma} \ref{lem: Gauss circle} {with $\Upsilon=\Upsilon_P$ and $n=d$}, we get that $N_P(R)=c_PR^d+o(R^d)$, where $c_P=\Vol(\Upsilon_P(1))$.

Now we show that $c_P$ is independent of the choice of $P$.  Let $\mathcal P$ denote the collection of $m$-dimensional positive definite subspaces $P$. Our claim will follow from the fact that $G = SO(Q)$ acts transitively on $\mathcal P$. Note that if $P \in \mathcal P$, $gP$ is also in $\mathcal P$ since $g$ preserves $Q$. This is due to Sylvester's law of inertia.

Next note that 
\begin{align*}
    \|v\|_{gP} &= \max\{ Q(v, u): u \in gP, Q(u, u) = 1\}\\
    &= \max\{ Q(v, gu): u \in P, Q(u, u) = 1\}\\
    &= \max\{ Q(g^{-1}v, u): u \in P, Q(u, u) = 1\}\\
    & = \|g^{-1}v\|_P.
\end{align*}
Therefore, 
\begin{align*}
    &\{v \in \mathbb R^d: Q(v,v)\geq 0, \|v\|_{gP} \le 1\} \\
    =&\{v \in \mathbb R^d: Q(v,v)\geq 0, \|g^{-1}v\|_P \le 1\}\\
    =&\{gw \in \mathbb R^d: Q(gw, gw)\geq 0, \|w\|_P \le 1\} \\
    =& \{gw \in \mathbb R^d: Q(w, w)\geq 0, \|w\|_P \le 1\} \\
    =& g\{w \in \mathbb R^d: Q(w,w)\geq 0, \|w\|_P \le 1\},
\end{align*}
and so $c_{gP} = c_P,$ since the action of $G = SO(Q)$ preserves Lebesgue measure on $\mathbb R^d$.
\end{proof}

\bigskip

\bibliographystyle{amsalpha}
\bibliography{AFL}

\medskip

\noindent Jayadev S. Athreya \\
\textsc{Department of Mathematics, University of Washington\\
Seattle, WA 98195, USA}\\
\texttt{jathreya@uw.edu}\\

\noindent Yu-Wei Fan \\
\textsc{Yau Mathematical Sciences Center, Tsinghua University\\
Beijing 100084, China}\\
\texttt{ywfan@mail.tsinghua.edu.cn}\\

\noindent Heather Lee \\
\texttt{heatherlee.math@gmail.com}\\

\end{document}